\documentclass[12pt,leqno]{article}
\usepackage{amsfonts}
\usepackage{amsmath, amsthm, amsfonts, amssymb, color}
\usepackage{mathrsfs}

\newtheorem{thm}{Theorem}[section]

\newtheorem{lem}[thm]{Lemma}

\newtheorem{exa}[thm]{Example}
\theoremstyle{definition}
\newtheorem{defn}{Definition}[section]
\newcommand{\scr}[1]{\mathscr #1}
\definecolor{wco}{rgb}{0.5,0.2,0.3}

\numberwithin{equation}{section} \theoremstyle{remark}
\newtheorem{rem}{Remark}[section]

\def\R{\mathbb R}
\def\ff{\frac}
\def\ss{\sqrt}
\def\B{\mathbf B}
\def\N{\mathbb N}
\def\kk{\kappa} 
\def\dd{\delta}  \def\vv{\varepsilon} \def\rr{\rho}
\def\<{\langle} \def\>{\rangle} \def\GG{\Gamma} \def\gg{\gamma}
    
\def\d{\text{\rm{d}}} \def\bb{\beta} \def\aa{\alpha} \def\D{\scr D}
  \def\si{\sigma} 
 \def\beq{\begin{equation}}  \def\F{\scr F}
 
\def\e{\text{\rm{e}}}  \def\OO{\Omega}  
 \def\tt{\tilde} 
 \def\P{\mathbb P} 
\def\C{\scr C}

\def\E{\mathbb E} 
  \def\LL{\Lambda}
 \def\B{\scr B}  
\def\to{\rightarrow}\def\ll{\lambda}
\def\8{\infty}  
\def\3{\triangle}\def\1{\lesssim}

\renewcommand{\bar}{\overline}

\renewcommand{\tilde}{\widetilde}

\newcommand{\barray}{\begin{array}{ll}}
\newcommand{\earray}{\end{array}}

\newcommand{\wdt}{\widetilde}
\newcommand{\bea}{\begin{displaymath}\begin{array}{rl}}
\newcommand{\eea}{\end{array}\end{displaymath}}

\title{Stationary Distributions for Retarded Stochastic Differential Equations without
Dissipativity \thanks{This research was supported in part by the National Science Foundation
under DMS-1207667.}}
\author{Jianhai Bao,\thanks{Department of Mathematics, Central South University,
Changsha, Hunan, 410075, P.R. China, majb@swansea.ac.uk}
 \and George
Yin,\thanks{Department of Mathematics, Wayne State University,
Detroit, MI 48202, USA, gyin@math.wayne.edu}  \and Chenggui
Yuan\thanks{Department of Mathematics, Swansea University,
Singleton Park, SA2 8PP, UK, C.Yuan@swansea.ac.uk}}

\begin{document}

\maketitle

\begin{abstract}
Retarded stochastic differential equations (SDEs) constitute a large
collection of systems arising in various real-life applications.
Most of the existing results make crucial use of dissipative
conditions. Dealing with ``pure delay'' systems in which both the
drift and the diffusion coefficients depend only on the
arguments with delays, the existing results become not applicable. This work
uses a variation-of-constants formula to overcome the difficulties
due to the lack of the information at the current time.
This paper establishes existence and
uniqueness of stationary distributions for  retarded
SDEs that need not satisfy dissipative conditions. The retarded SDEs
considered in this paper also cover SDEs of neutral type and SDEs
driven by L\'{e}vy processes that might not admit finite second
moments.

\vskip 0.2 true in \noindent {\bf Keywords:} retarded stochastic
differential equation,  stationary distribution,
variation-of-constants formula, neutral type equation, L\'{e}vy process.

\vskip 0.2 true in \noindent

 \end{abstract}

\newpage

\section{Introduction}
Retarded stochastic differential equations (SDEs) are such SDEs
involving retarded arguments. For numerous systems involving delays
and arising in real world applications such as in life insurance,
risk management, wireless communication, and optimal control of
multi-agent systems, one has to take long-term performance into
consideration. Thus, an important problem concerning retarded SDEs
is the existence of stationary distributions (see Definition
\ref{stationary} in what follows).

\smallskip

One of the main approaches in
the literature to date is to incorporate certain dissipativity to establish
the existence of stationary distributions
for retarded SDEs. The dissipativity is normally assured by imposing
 information of the current time with certain decay conditions.
Such an idea has been used extensively. For instance,
utilizing the remote start method (i.e., dissipative method), Bao et
al. \cite{BYY}
 discussed several class of retarded SDEs which include  SDEs with variable
 delays  and SDEs with jumps; applying an exponential-type estimate, Bo and
Yuan
 \cite{BY13} investigated retarded SDEs driven by Poisson jump
 processes; adopting the Arzel\`{a}--Ascoli tightness characterization, Es-Sarhir et
al. \cite{ESV} and Kinnally and Williams \cite{KW} considered
 retarded SDEs with super-linear drift terms and positivity constraints
 respectively. However, the existing literature
 cannot deal with the following seemingly simple linear retarded SDE on the real line $\R$,
\begin{equation}\label{A1}
\d X(t)=-X(t-1)\d t+\si X(t-1)\d W(t),\ \ \ X_0=\xi,
\end{equation}
where $\si\in\R$ and $\{W(t)\}_{t\ge0}$ is a real-valued standard
Brownian motion.  Observe that it is impossible to choose
$\ll_1>\ll_2>0$ such that
\begin{equation*}
-2xy+\si^2y^2\le -\ll_1x^2+\ll_2y^2,\ \ \ x,y\in\R
\end{equation*}
holds for some appropriate $\si\in\R.$ That is, \eqref{A1} does not
obey a dissipative condition. Therefore, the techniques used in
\cite{BYY,BY13,ESV,KW} are not
applicable to
\eqref{A1}. As can be seen that the main problem is because of the lack of the information
at the current time, so no dissipative conditions can be used.

In the
well-known work  \cite{York}, Yorke treated deterministic
systems with pure delays. His work has stimulated much of the subsequent work resulting
in a vast literature on the pure delay equations in the deterministic setup.
Our consideration in this paper is a
generalization of the model in \cite{York} in that both the drift
and diffusion coefficients involve only retarded elements. The
right-hand sides of such differential equations do not involve
information on current time. Consequently, it is not possible to use
any dissipative conditions.
With regard to uniqueness of stationary
distributions of retarded SDEs, by an asymptotic coupling method,
Hairer et al. \cite{HMS} discussed a wide range of  non-degenerate
retarded SDEs under some mild assumptions,
 which   need not
guarantee existence of a stationary  distribution, however.  Scheutzow
\cite{S12}
 studied a very simple
linear retarded SDE without the drift term.

\smallskip
In this paper, we aim to obtaining existence and uniqueness of
stationary distributions  for a class of retarded SDEs, which
includes \eqref{A1} as a special case. The key word is that we focus
on systems {\it without satisfying dissipative conditions}. To
overcome the difficulties,
we
use the variation-of-constants formula, which has been applied
successfully in Gushchin and K\"uchler \cite{GK} and Liu \cite{L10}.
In the aforementioned references, the authors
considered
stationary solutions
 (see Remark \ref{remark}), for finite-dimensional and infinite-dimensional retarded
Ornstein-Uhlenbeck (O-U) processes,
  respectively.
It is also worth pointing out that the variation-of-constants
formula together with the semi-martingale characteristics has been
utilized to study the existence of stationary distributions for a
class of retarded SDEs, which  do not include \eqref{A1}, however.

\smallskip
The
remainder of the paper is organized as follows. Section \ref{sec}
reviews the variation-of-constants formula for deterministic linear
retarded systems and collects some auxiliary lemmas. Section
\ref{Sec2} is devoted to existence and uniqueness of stationary
distributions for semi-linear retarded SDEs by the
variation-of-constants formula and the stability-in-distribution
approach. Section \ref{Sec3} generalizes the theory established in
Section \ref{Sec2} to  SDEs of neutral type. The last section
focuses on  existence and uniqueness of stationary distributions for
retarded SDEs driven by L\'{e}vy processes, which need not have
finite second moments. The key there is to use the variation-of-constants formula
together with the
 tightness criterion due to Kurtz.

\section{Preliminary}\label{sec}
We start with some terminologies and notation. Let $(\OO,\P,\F)$ be
a probability space together with a filtration
$\{\F_t\}_{t\ge0}$ satisfying the usual conditions (i.e.,
$\F_{t_+}:=\cap_{s>t}\F_s=\F_t$, $\F_s \subset \F_t$ for $s\le t$, and $\F_0$ contains all $\P$-null
sets).
   Let
$\{W(t)\}_{t\ge0}$ and $\{Z(t)\}_{t\ge0}$ be  real-valued Brownian
motion and  L\'{e}vy process defined on the
stochastic basis $(\OO,\P,\F,\{\F_t\}_{t\ge0})$, respectively,
For each $t\ge0,$ $Z(t)$ is
infinitely divisible
by virtue of \cite[Proposition 1.3.1, p.43]{A09}. By
the L\'{e}vy--Khintchine formula \cite[Theorem 1.2.14, p.29]{A09},
the symbol or characteristic exponent of $Z(t)$
satisfies
\begin{equation*}
\Psi(\xi):=\ff{1}{2}a\xi^2+ib\xi+\int_{z\neq0}\{1-\e^{-\xi z}+i\xi
z
{{\mathbf 1}}_{\{|z|\le1\}}\}\nu(\d z),
\end{equation*}
in which $a\ge0$, $b\in\R$ and $\nu(\cdot)$ is a L\'{e}vy measure,
i.e., a $\si$-finite measure on $\R\setminus\{0\}$ such that
$$\int_{z\neq0}(1\wedge|z|^2)\nu(\d z)<\8.$$ Fix $\tau\in(0,\8)$,
which is referred to as the delay. Recall that a path $f:
[-\tau,0]\mapsto\R$ is called c\'{a}dl\'{a}g if it is
right-continuous having finite left-hand limits.
For
a subinterval $U\subset(-\8,\8)$, $C(U;\R)$ (resp. $D(U;\R)$) denotes the
family of all  real-valued continuous (resp. c\'{a}dl\'{a}g) functions
defined on $U$. Let $\C:=C([-\tau,0];\R^n)$ equipped with the
uniform metric
$\|\zeta\|_\8:=\sup_{-\tau\le\theta\le0}|\zeta(\theta)|$ for
$\zeta\in\C$, and $\D:=D([-\tau, 0];\R)$ endowed with the Skorohod
metric
\begin{equation*}
\d_S(\xi,\eta):=\inf_{\ll\in\LL}\{\|\ll\|^\circ\vee\|\xi-\eta\circ\ll\|_\8\},\
\ \ \xi,\eta\in\D.
\end{equation*}
Here $\LL$ is the class of increasing homeomorphisms satisfying
\begin{equation*}
\|\ll\|^\circ:=\sup_{-\tau\le
s<t\le0}\Big|\log\ff{\ll(t)-\ll(s)}{t-s}\Big|<\8,
\end{equation*}
and $\eta\circ\ll$ means the composition of mappings $\eta$ and
$\ll$. Under the uniform metric $\|\cdot\|_\8$, the space $\D$ is
complete but not separable,  whereas, under the Skorohod metric
$\d_S$, $\D$ is complete and separable (see, e.g., \cite[Theorem
12.2, p.128]{B}). For more details on the Skorohod metric, we refer
to \cite[Chapter 4]{B}. For a continuous (resp. c\'{a}dl\'{a}g) function
$f:[-\tau,\8)\mapsto\R$ with $t\ge0$, let $f_t\in\C$ (resp. $f_t\in\D$) be such
that $f_t(\theta)=f(t+\theta),\theta\in[-\tau,0]$. As usual,
$\{f_t\}_{t\ge0}$ is called the segment process of
$\{f(t)\}_{t\ge-\tau}$.  The
 notation $\mathcal {P}(\C)$ (resp. $\mathcal {P}(\D)$) denotes the family of all probability
measures on $(\C,\B(\C))$ (resp. $(\D,\B(\D))$ and $\B_b(\C)$ (resp. $\B_b(\D))$
denotes
the set
of bounded and continuous functions $F:\C \mapsto \R$ (resp.  bounded measurable functions  $F:\D\mapsto\R$) endowed
with the uniform norm $\|F\|_0:=\sup|F(\phi)|$.
Use $\mu(\cdot)$ and
$\rr(\cdot)$ to denote the finite signed measures defined on
$[-\tau,0]$. Let $\mathbb{C}$ be the set of all complex numbers and
$\mbox{Re}(z)$ stand for the real part of $z\in\mathbb{C}$.
Throughout this paper, $c>0$ is used as a generic
positive constant whose values
may change for different usage.

\smallskip
By the variation-of-constants formula (see, e.g., \cite[Theorem 1.2,
p.170]{HL}), the following linear retarded equation
\begin{equation}\label{a1}
\d Y(t)=\int_{-\tau}^0Y(t+\theta)\mu(\d\theta)\d t,\ \ \
Y_0=\xi\in\C
\end{equation}
has a unique explicit
representation of the solution
\begin{equation*}
Y(t;\xi)=r(t)\xi(0)+\int_{-\tau}^0\int_\theta^0
r(t+\theta-s)\xi(s)\d s\mu(\d \theta),
\end{equation*}
where $r(t)$ is the fundamental solution of \eqref{a1} with the
initial data $r(0)=1$ and $r(\theta)=0$ for $\theta\in[-\tau,0).$
Let
\begin{equation*}
v_0:=\sup\{\mbox{Re}(\ll):\ll\in \mathbb{C},\ \3(\ll)=0\},
\end{equation*}
where
\begin{equation*}
\triangle(\ll):=\ll-\int_{-\tau}^0\e^{\ll s}\mu(\d s),\ \ \ \ll\in
\mathbb{C}.
\end{equation*}
According to Hale and Verduyn Lunel \cite{HL}, $\3(\ll)=0$ is called
the characteristic equation of equation \eqref{a1}. Then, by virtue
of
\cite[Theorem 3.2, p.271]{HL}, for any $\gg>v_0$, there exists
$c=c(\gg)>0$ such that
\begin{equation}\label{eq3}
|r(t)|\le ce^{\gg t},\ \ \ t\ge-\tau.
\end{equation}
For more details on the variation-of-constants formula of general
retarded linear systems, we refer to \cite[Chapter 6 and Chapter
9]{HL}.

\smallskip

Before the end of this section, we collect some preliminary lemmas
for
later use. The first lemma is a generalized Gronwall
inequality and the second one is concerned with Kurtz's criterion on
tightness of laws on $\D$.

\begin{lem}\label{Inq2}
{\rm (\cite[Lemma 8.2]{IN}) Let $u:[0,\8)\mapsto\R_+$ be a
continuous function and $\dd>0, \aa>\bb>0.$ If
\begin{equation*}
u(t)\le\dd+\bb\int_0^t\e^{-\aa(t-s)}u(s)\d s,\ \ \ t\ge0,
\end{equation*}
then $u(t)\le(\dd\aa)/(\aa-\bb)$. }
\end{lem}

\begin{lem}\label{tightness}
{\rm (\cite[Theorem 8.6, p.137-138]{EG}) For each $t\ge0$, let
$Y^t(\cdot)\in D([0,\tau];\R)$ and assume that
\begin{equation*}
\lim_{K_1\to\8}\limsup_{t\to\8}\P\Big(\sup_{0\le s\le T}|Y^t(s)|\ge
K_1\Big)=0 \mbox{ for each }T\le\tau,
\end{equation*}
and, for all $0\le u\le\dd,v\le T$, that there exists $\bb>0$ such
that
\begin{equation*}
\begin{split}
&\E^t_v\{|Y^t(u+v)-Y^t(u)|\wedge1\}\le\E^t_v\gg(t,\dd)\\
&\lim_{\dd\to0}\limsup_{t\to\8}\E\gg(t,\dd)=0,
\end{split}
\end{equation*}
where $\E^t_v$ denotes the conditional expectation with respect to
$\F^t_v$, the minimal $\si$-algebra that measures $\{Y^t(s)\}_{0\le
s\le v}$. Then $\{\mathcal {L}(Y^t(s)),s\in[0,\tau] \}_{t\ge0}$ is
tight in $D([0,\tau];\R)$, where $\mathcal {L}(\eta)$ means the law
of random variable $\eta.$ }
\end{lem}

\section{Stationary Distributions for  Retarded
SDEs}\label{Sec2} In this section, we consider a  semi-linear
retarded SDE
of the form
\begin{equation}\label{a5}
\d X(t)=\Big(\int_{-\tau}^0X(t+\theta)\mu(\d\theta)\Big)\d
t+\si(X_t)\d W(t),\ \ \ X_0=\xi\in\C,
\end{equation}
where $\si:\C\mapsto\R$ is Borel measurable and there exists an
$L>0$ such that
\begin{equation}\label{a6}
|\si(\xi)-\si(\eta)|^2\le
L\Big(|\xi(0)-\eta(0)|^2+\int_{-\tau}^0|\xi(\theta)-\eta(\theta)|^2\rho(\d\theta)\Big),
\ \ \ \xi,\eta\in\C,
\end{equation}
where $\rr(\cdot)$ is a measure on $[-\tau,0].$ Under \eqref{a6}, by
\cite[Theorem 2.1, p.36]{M84}, \eqref{a5} admits  a unique strong
solution $\{X(t;\xi)\}_{t\ge-\tau}$ with the initial segment
$X_0=\xi\in\C.$ Throughout this section, we further assume that the
initial segment $X_0=\xi\in\C$ is independent of $\{W(t)\}_{t\ge0}$.
It seems to be more instructive to present the main line of argument without undue complicated notation.
Thus, we will focus only on  real-valued retarded SDEs  in this paper.

\smallskip

Before stating our main result in this section, we recall the notion
of the stationary distribution (see, e.g., \cite[Definition
2.2.1]{KW}).

\begin{defn}\label{stationary}
{\rm A stationary distribution for \eqref{a5} is a probability
measure $\pi\in\mathcal {P}(\C)$ such that
\begin{equation*}
\pi(F)=\pi(P_tF),\ \ \ t\ge0,
\end{equation*}
where $$\pi(F):=\int_{\C}F(\xi)\pi(\d\xi)$$ and
$$P_tF(\xi):=\E
F(X_t(\xi)) \ \hbox{ for each } \ F\in\B_b(\C).$$ }
\end{defn}

\begin{rem}
{\rm If $\pi\in\mathcal {P}(\C)$ is a stationary distribution of
\eqref{a5} and the initial segment enjoys the same law, by
\cite[Lemma 1.1.9, p.14]{A09}, the independence of $\xi\in\C$ and
$\{W(t)\}_{t\ge0}$ and the smooth property of conditional
expectation,
 one has
\begin{equation*}
\pi(F)=\int_\C\E
F(X_t(\eta))\pi(\d\eta)=\E(\E(F(X_t(\xi)))|\F_0)=\E(F(X_t(\xi))).
\end{equation*}
Then we conclude that $X_t(\xi)$ shares the law $\pi\in\mathcal
{P}(\C)$, i.e., the law of $X_t(\xi)$ is invariant under time
translation. The main result of this section is stated next.
The approach that we are using is based on the idea of stability in distribution
argument.

}
\end{rem}

\begin{thm}\label{diffusion}
{\rm Let $v_0<0$ and assume further that \eqref{a6} holds for a
sufficiently small $L>0$. Then \eqref{a5} has a unique stationary
distribution $\pi\in\mathcal {P}(\C)$. }
\end{thm}

\begin{proof}
We adopt the stability-in-distribution approach;
see for example, \cite[Theorem
3.2]{YZM}.
Mainly we need to verify the following two conditions hold:
\begin{enumerate}
\item[($\N1$)] $\lim_{t\to\8}\sup_{\xi,\eta\in
U}\E\|X_t(\xi)-X_t(\eta)\|_\8^2=0$;
\item[($\N2$)] $\sup_{t\ge0}\sup_{\xi\in U}\E\|X_t(\xi)\|_\8^2<\8$,
\end{enumerate}
where $U$ is a bounded subset of $\C,$ then $\P(t,\xi,\cdot)$.
Under the aforementioned conditions, the
transition kernel of $X_t(\xi)$, converges weakly to $\pi\in\mathcal
{P}(\C)$. For any $F\in C_b(\C)$, the set of all bounded and continuous
real-valued functions on $\C$, by the Markovian property of
$\{X_t(\xi)\}_{t\ge0}$ (see, e.g., \cite[Theorem 1.1, p.51]{M84}),
one has
\begin{equation*}
P_{t+s}F(\xi)=P_sP_tF(\xi),\ \ \ t,s\ge0.
\end{equation*}
Then, for fixed $t\ge0$, taking $s\to\8$ gives that
\begin{equation*}
\pi(F)=\pi(P_tF)
\end{equation*}
whenever $\P(t,\xi,\cdot)$ converges weakly to $\pi\in\mathcal
{P}(\C)$. Hence, \eqref{a5} admits a stationary distribution
$\pi\in\mathcal {P}(\C)$ provided that $(\N1)$ and $(\N2)$ hold.
In what follows, it
suffices to claim that $(\N1)$ and $(\N2)$ are fulfilled
respectively. According to \cite[Theorem 3.1]{RR}, the unique strong
solution $\{X(t;\xi)\}_{t\ge0}$ can be represented explicitly by
\begin{equation}\label{a7}
X(t;\xi)=r(t)\xi(0)+\int_{-\tau}^0\int_\theta^0
r(t+\theta-s)\xi(s)\d s\mu(\d \theta)+\int_0^tr(t-s)\si(X_s(\xi))\d
W(s),
\end{equation}
in which $\{r(t)\}_{t\ge-\tau}$ is the fundamental solution of
\eqref{a1}.
To proceed, take the difference of the two solutions
with different initial segments  $$\Xi(t;\xi,\eta):=X(t;\xi)-X(t;\eta).$$ By
\eqref{a6},
\eqref{a7},
and the It\^o
isometry imply that
\begin{equation}\label{a2}
\begin{split}
\E&|\Xi(t;\xi,\eta)|^2\\&\le
3\Big\{|r(t)(\xi(0)-\eta(0))|^2+\Big|\int_{-\tau}^0\int_\theta^0
r(t+\theta-s)(\xi(s)-\eta(s))\d s\mu(\d \theta)\Big|^2\\
&\quad+\Big|\int_0^tr(t-s)(\si(X_s(\xi))-\si(X_s(\eta)))\d
W(s)\Big|^2\Big\}\\
&\le
c\Big\{\e^{-2\gg t}\|\xi-\eta\|^2_\8\\
&\quad+L\int_0^t\e^{-2\gg(t-s)}\Big(|\Xi(s;\xi,\eta)|^2+\int_{-\tau}^0|\Xi(s+\theta;\xi,\eta)|^2\rr(\d\theta)\Big)\d
s\Big\}\\
&\le c\Big\{\e^{-2\gg
t}\|\xi-\eta\|^2_\8+L\int_0^t\e^{-2\gg(t-s)}|\Xi(s;\xi,\eta)|^2\d
s\Big\}.
\end{split}
\end{equation}
Multiplying by $\e^{2\gg t}$ on both sides of \eqref{a2} gives that
\begin{equation*}
\e^{2\gg t}\E|\Xi(t;\xi,\eta)|^2\le
c\Big\{\|\xi-\eta\|^2_\8+L\int_0^t\e^{2\gg s}|\Xi(s;\xi,\eta)|^2\d
s\Big\}.
\end{equation*}
Recall that the $c$ above is a generic positive constant. So, the
Gronwall inequality leads to
\begin{equation}\label{b4}
\E|\Xi(t;\xi,\eta)|^2\le c\|\xi-\eta\|_\8^2\e^{-\aa t},
\end{equation}
where $\aa:=2v_0-cL>0$ since $L>0$ is sufficiently small. By the
H\"older inequality and the Burkhold-Davis-Gundy inequality (see,
e.g., \cite[Theorem 7.3, p.40]{M08}, we obtain from \eqref{a5},
\eqref{a6} and \eqref{b4} that
\begin{equation}\label{b5}
\begin{split}
\E\|\Xi_t(\xi,\eta)\|_\8^2&\le
3\E|\Xi(t-\tau;\xi,\eta)|^2+3\tau\E\int_{t-\tau}^t\Big|\int_{-\tau}^0\Xi(s+u;\xi,\eta)\mu(\d
u)\Big|^2\d s\\
&\quad+3\E\Big(\sup_{-\tau\le\theta\le0}\Big|\int_{t-\tau}^{t+\theta}(\si(X_s(\xi))-\si(X_s(\eta)))\d
W(s)\Big|^2\Big)\\
&\le c  \|\xi-\eta\|_\8^2\e^{-\aa t}.
\end{split}
\end{equation}
Hence  ($\N1$) holds. In
what follows, we show that ($\N2$) is also
valid. Applying  \eqref{a6} and \eqref{a7}, and utilizing the It\^o
isometry yield that
\begin{equation*}
\begin{split}
\E|X(t;\xi)|^2&\le3\Big\{|r(t)\xi(0)|^2+\Big|\int_{-\tau}^0\int_\theta^0
r(t+\theta-s)\xi(s)\d s\mu(\d
\theta)\Big|^2\\
&\quad+2\E\int_0^t|r(t-s)(\si(X_s(\xi))-\si(0))|^2\d
s+2\E\int_0^t|r(t-s)\si(0)|^2\d s\Big\}\\
&\le c\Big\{\e^{-2\gg
t}\|\xi\|^2_\8+\int_0^t\e^{-2\gg(t-s)}|\si(0)|^2\d
s\Big\}\\
&\quad+c_0L\E\int_0^t\e^{-2\gg(t-s)}\Big(|X(s;\xi)|^2+\int_{-\tau}^0|X(s+\theta;\xi)|^2\rr(\d\theta)\Big)|^2\d
s\\
&\le c+2c_0 L\E\int_0^t\e^{-2\gg(t-s)}|X(s;\xi)|^2\d s,
\end{split}
\end{equation*}
where $c_0>0$ is some constant. By Lemma \ref{Inq2}, we arrive at
\begin{equation}\label{a3}
\sup_{t\ge0}\E|X(t;\xi)|^2\le c
\end{equation}
since $L>0$ is sufficiently small. Following a similar argument to
that of \eqref{b5} and taking \eqref{a3} into account, one has
\begin{equation}\label{A2}
\sup_{t\ge0}\E\|X_t(\xi)\|_\8^2\le c
\end{equation}
and therefore ($\N2$) holds. \eqref{A2}, in addition to \eqref{b5},
implies that
\begin{equation}\label{b6}
\E\|X_t(\xi)\|_\8^2\le2\E\|X_t(\xi)-X_t(0)\|_\8^2+\E\|X_t(0)\|_\8^2\le
c(1+\e^{-\aa t}\|\xi\|_\8^2).
\end{equation}
Employing the invariance of $\pi\in\mathcal {P}(\C)$ and integrating
with respect to $\pi(\cdot)\in\mathcal {P}(\C)$ on both sides of
\eqref{b6} lead to
\begin{equation}\label{eq16}
\pi(\|\cdot\|_\8)<\8.
\end{equation}
If $\pi^\prime(\cdot)\in\mathcal {P}(\C)$ is also a stationary
distribution of \eqref{a5}, for any bounded Lipschitz function
$f:\C\mapsto\R$, by \eqref{b5} and the invariance of
$\pi(\cdot),\pi^\prime(\cdot)\in\mathcal {P}(\C)$, it follows from
\eqref{eq16} that
\begin{equation}\label{eq18}
|\pi(f)-\pi^\prime(f)|\le\int_{\C\times\C}| P_tf(\xi)-
P_tf(\eta)|\pi(\d\xi)\pi^\prime(\d\eta)\le c\e^{-\ll t},\ \ \ \
t\ge0.
\end{equation}
This implies the uniqueness of stationary distribution by taking
$t\to\8$ in \eqref{eq18}.
\end{proof}

\begin{rem}
{\rm By the invariance of $\pi\in\mathcal {P}(\C)$, for any
$F\in\B_b(\C)$,
\begin{equation}\label{eq17}
 |P_tF(\xi)-\pi(F)|\le\int_\C| P_tF(\xi)-P_tF(\eta)|\pi(\d\eta).
\end{equation}
Thus,  taking \eqref{b5},  \eqref{eq16}, \eqref{eq17}, and
\cite[Lemma 7.1.5, p.125]{DZ} into consideration yields that
\begin{equation*}
|P_tF(\xi)-\pi(F)|\le c\e^{-\aa t}
\end{equation*}
for arbitrary $t\ge0,\ F\in\B_b(\C),\ \xi\in\C$ and some  $\aa>0.$
That is, $\{P_t\}_{t\ge0}$ converges exponentially to the
equilibrium uniformly with respect to $\xi$ in each ball with
finite radius in $\C$. }
\end{rem}

\begin{rem}
{\rm Under dissipative conditions, by the  Arzel\`{a}--Ascoli
tightness characterization, Es-Sarhir et al. \cite{ESV} and Kinnally
and Williams \cite{KW} exploited existence of stationary
distributions of retarded SDEs with super-linear drift terms and
positivity constraints, respectively.  Applying the It\^o formula,
they gave the uniform boundedness, which plays a key role in
analyzing the diffusion terms by the Kolmogrov tightness criterion,
for higher moments of the segment processes. However, \eqref{a5}
need not satisfy any dissipative conditions, and therefore the
tricks adopted in Es-Sarhir et al. \cite{ESV} and Kinnally and
Williams \cite{KW} no longer work. In this section, adopting the
variation-of-constants formula and utilizing the
stability-in-distribution technique, we provide verifiable criterion
to capture a unique stationary distribution for a class of
semi-linear retarded SDEs, where, in particular, the characteristic
equation of the corresponding deterministic counterpart play a key
role.

 }
\end{rem}

\begin{rem}\label{remark}
{\rm Under dissipative conditions, It\^o and Nisio \cite{IN}
discussed existence of stationary solutions for retarded SDEs. By
\cite[Theorem 3]{IN}, we deduce from \eqref{b6} that \eqref{a5},
without satisfying a dissipative condition, has a stationary
solution. A solution $\{X(t)\}_{t\ge-\tau}$ of \eqref{a5} is called
strong stationary, or simply stationary, if the finite-dimensional
distributions are invariant under time translation, i.e.,
\begin{equation*}
\P\{X(t+t_k)\in\GG_k,\ k=1,\ldots,n\}=\P\{X(t_k)\in\GG_k,\
k=1,\ldots,n\}
\end{equation*}
for all $t\ge0,t_k\ge-\tau$ and $\GG_k\in\B(\R)$. For stationary
solutions of retarded O-U processes in Hilbert spaces, we refer to,
e.g.,  Liu \cite{L10}. It is worth   pointing out that the stationary
solutions discussed for example, in  It\^o and Nisio \cite{IN} and Liu
\cite{L10} are related to the solution process
$\{X(t)\}_{t\ge-\tau}$. Nevertheless, the stationary distributions
in our case
involve  the segment process $\{X_t\}_{t\ge0}$. }
\end{rem}

Before concluding this section,  we give an example
to show the validity
of Theorem \ref{diffusion}. Note that the desired results in the following
example
cannot be
obtained by any of the results in Bao et al. \cite{BY},  Es-Sarhir et al.
\cite{ESV}, and Kinnally and Williams \cite{KW}.

\begin{exa}
{\rm  Consider a semi-linear retarded SDE
\begin{equation}\label{A3}
\d X(t)=-X(t-1)\d t+\si(X(t-1))\d W(t),\ \ \ X_0=\xi\in\C,
\end{equation}
It is readily
seen that the corresponding   characteristic equation is
\begin{equation}\label{a11}
\ll+\e^{-\ll}=0.
\end{equation}
A simple calculation
using Matlab yields that the unique root of \eqref{a11} is
$$\ll=-0.3181 + 1.3372\mbox{i}.$$ Thus, by Theorem \ref{diffusion},
we deduce that  \eqref{A3} possesses a unique stationary
distribution $\pi\in\mathcal {P}(\C)$ whenever the Lipschitz
constant of $\si:\R\to\R$ is sufficiently small. Hence, \eqref{A1}
also has a unique stationary distribution $\pi\in\mathcal {P}(\C)$
for a sufficiently small Lipschitz constant. Note that \eqref{A3}
  does not satisfy a dissipative condition even
the Lipschitz constant of $\si$ is sufficiently small since it is
impossible to choose constants $\ll_1>\ll_2>0$ such that
\begin{equation*}
2xy+\si^2(y)\le c-\ll_1|x|^2+\ll_2|y|^2,\ \ \ x,y\in\R.
\end{equation*}
Therefore, \eqref{A3}  cannot be covered by Bao et al. \cite[Theorem
3.2]{BY}, Es-Sarhir et al. \cite{ESV} and Kinnally and Williams
\cite{KW}. }
\end{exa}

The following example shows that, in some cases,  the
variation-of-constants technique and the dissipative method adopted
in, e.g., Bao et al. \cite[Theorem 3.2]{BY} play the same role in
the exploration of existence and uniqueness of stationary
distributions.

\begin{exa}\label{example}
{\rm Consider a semi-linear retarded SDE
\begin{equation}\label{a9}
\d X(t)=\{a X(t)+bX(t-1)\}\d t+\si(X(t-1))\d W(t),\ \ \
X_0=\xi\in\C,
\end{equation}
where $a<0,b\in\R$ and $\si:\R\to\R$ is Lipschitzian with a
Lipschitz constant sufficiently small. For this case, it is trivial
to see that $\mu(\cdot)=a \dd_0(\cdot)+b\dd_{-1}(\cdot),$ where
$\dd$ is the Dirac measure. Note that the characteristic equation
corresponding to the deterministic counterpart of \eqref{a9} is
\begin{equation}\label{a10}
\ll-a-b\e^{-\ll}=0.
\end{equation}
By \cite[Theorem 1]{LT}, all the roots of \eqref{a10} have negative
real parts if and only if
\begin{equation}\label{eq21}
a<b<-a.
\end{equation}
Then, by Theorem \ref{diffusion}, \eqref{a9} admits a unique
stationary distribution $\pi\in\mathcal {P}(\C)$ provided that the
Lipschitz constant of $\si$ is sufficiently small. Furthermore, for
any $x,y\in\R$, by the elemental inequality: $2uv\le\vv
u^2+\vv^{-1}v^2,u,v\in\R,\vv>0$, we obtain
\begin{equation}\label{eq20}
2x(ax+by)=2ax^2+2bxy\le-(-2a-\vv)x^2+\ff{b^2}{\vv}y^2,\ \ \ \vv>0.
\end{equation}
In particular, taking $\vv=|b|$ in \eqref{eq20} leads to
\begin{equation*}
2x(ax+by)=2ax^2+2bxy\le-(-2a-|b|)x^2+|b|y^2.
\end{equation*}
If $-2a-|b|>|b|$, i.e., \eqref{eq21} holds, then \eqref{a9}
satisfies a dissipative condition whenever the Lipschitz constant of
$\si$ is sufficiently small. Consequently, \cite[Theorem 3.2]{BY}
also yields that  \eqref{a9} has a unique stationary distribution. }
\end{exa}

\section{Stationary Distributions for SDEs of Neutral Type}\label{Sec3}
In this section,  we proceed to generalize Theorem \ref{diffusion}
to SDEs of neutral type. To begin,
we give an overview of the
variation-of-constants formula for linear equations of neutral type.
By \cite[Theorem 1.1, p.256]{HL}, the following linear equation of
neutral type
\begin{equation}\label{N2}
\d
\Big(Y(t)-\int_{-\tau}^0Y(t+\theta)\rr(\d\theta)\Big)=\Big(\int_{-\tau}^0Y(t+\theta)\mu(\d\theta)\Big)\d
t
\end{equation}
with the initial
data $\xi\in\C$  has a unique solution
$\{Y(t;\xi)\}_{t\ge-\tau}$.   By virtue of \cite[Theorem 2.2]{L},
for any $\xi\in\C$ such that
$\int_{-\tau}^0|\xi^\prime(\theta)|^2\d\theta<\8$, $Y(t;\xi)$ can be
expressed explicitly by
\begin{equation*}
\begin{split}
Y(t;\xi)
=r(t)\xi(0) &-\int_{-\tau}^0 r(t+\theta)\xi(0)\rr(\d
\theta)+\int_{-\tau}^0\int_s^0r(t+\theta-s)\xi(s)\d s\mu(\d
\theta)\\
&
+\int_{-\tau}^0\int_s^0 r(t-s+\theta)\xi^\prime(s)\d s\rr(\d
\theta),
\end{split}
\end{equation*}
 where $r(t)$ is the fundamental solution of \eqref{N2} with the
initial segment $r(0)=1$ and $r(\theta)=0,\theta\in[-\tau,0).$ Let
\begin{equation*}
\bar v_0:=\sup\{\mbox{Re}(\ll):\ll\in \mathbb{C},\ \3_0(\ll)=0\},
\end{equation*}
where
\begin{equation*}
\3_0(\ll):=\ll-\ll\int_{-\tau}^0\e^{\ll\theta}\rr(\d\theta)-\int_{-\tau}^0\e^{\ll\theta}\mu(\d\theta),\
\ \ \ll\in\mathbb{C}.
\end{equation*}
In view of \cite[Theorem 3.2, p.271]{HL}, one has
\begin{equation}\label{N3}
|G(t)|\le c\e^{\aa t}\ \ \ \mbox{ for any } \aa>\bar v_0.
\end{equation}
For more details on the variation-of-constants formula  of equations
of neutral type, we refer the reader to \cite[Chapeter 9]{HL}.

\smallskip

In this section, we consider a semi-linear SDE of neutral type  in
the form
\begin{equation}\label{N1}
\d
\Big(X(t)-\int_{-\tau}^0X(t+\theta)\rr(\d\theta)\Big)=\Big(\int_{-\tau}^0X(t+\theta)\mu(\d\theta)\Big)\d
t+\si(X_t)\d W(t)
\end{equation}
with the initial value $X_0=\xi\in\C$, where
 $\si(\cdot):\C\to\R$ such that \eqref{a6} and $\{W(t)\}_{t\ge0}$ is a real-valued
Brownian motion  defined on the probability space
$(\OO,\P,\F,\{\F_t\}_{t\ge0})$.

\smallskip

Our main result in this section is presented as follows.

\begin{thm}\label{neutral}
{\rm Let $\bar v_0<0,\kk:=\mbox{Var}(\rr)<1/2$ and assume further
that \eqref{a6} holds for a sufficiently small $L>0$. Then
\eqref{N1} has a unique stationary distribution $\pi\in\mathcal
{P}(\C)$. }
\end{thm}

\begin{proof}
By a close inspection of the argument of  Theorem \ref{diffusion},
for a bounded subset $U\subset\C$, it is sufficient to show that
\begin{equation}\label{N7}
\E\|X(t;\xi)-X(t;\eta)\|^2_\8\le c\e^{-\aa t},\ \ \ \xi,\eta\in U
\end{equation}
for some $\aa>0$, and
\begin{equation}\label{N8}
\sup_{t\ge0}\E\|X(t;\xi)\|^2_\8<\8,\ \ \ \xi\in U.
\end{equation}
In what follows, we assume that $\xi,\eta\in C^2_b(\C)$ without any
confusions. By the variation-of-constants formula \cite[Theorem
3.1]{RR}, \eqref{N1} can be written as
\begin{equation}\label{N4}
\begin{split}
X(t;\xi)=r(t)\xi(0) &-\int_{-\tau}^0 r(t+\theta)\xi(0)\rr(\d
\theta)+\int_{-\tau}^0\int_s^0r(t+\theta-s)\xi(s)\d s\mu(\d
\theta)\\
& +\int_{-\tau}^0\int_s^0 r(t-s+\theta)\xi^\prime(s)\d s\rr(\d
\theta)+\int_0^tr(t-s)\si(X_s(\xi))\d W(s).
\end{split}
\end{equation}
Carrying out arguments   analogous to that of \eqref{a2} and
\eqref{a3} respectively, for any $\xi,\eta\in C^2_b(\C)$, we derive
from \eqref{N4} that
\begin{equation}\label{N5}
\E|X(t;\xi)-X(t;\eta)|^2\le c\e^{-\aa t},\ \ \ t\ge0
\end{equation}
for some $\aa>0$, and
\begin{equation}\label{N6}
\sup_{t\ge0}\E|X(t;\xi)|^2<\8
\end{equation}
whenever $L>0$ is sufficiently small. In
terms of \eqref{N5} and the Burkhold-Davis-Gundy inequality (see,
e.g., \cite[Theorem 7.3, p.40]{M08}), one obtains from \eqref{a6}
and \eqref{N1} that
\begin{equation*}
\E\Big(\sup_{t-\tau\le s\le
t}\Big|X(s;\xi)-X(s;\eta)-\int_{-\tau}^0(X(s+\theta;\xi)-X(s+\theta;\eta))\rr(\d\theta)\Big|^2\Big)\le
c\e^{-\aa t},\ \ \ t\ge0.
\end{equation*}
By the
elementary inequality:
\begin{equation*}
(a+b)^2\le a^2/(1-\vv)+b^2/\vv,\ \ \ a,b\in\R,\ \vv\in(0,1),
\end{equation*}
for any integer $n\ge1$, it thus follows that
\begin{equation*}
\begin{split}
&\E\|X_{n\tau}(\xi)-X_{n\tau}(\eta)\|_\8^2\\
&\le\ff{1}{\kk}\E\Big(\sup_{(n-1)\tau\le s\le
n\tau}\Big|\int_{-\tau}^0(X(s+\theta;\xi)-X(s+\theta;\eta))\rr(\d\theta)\Big|^2\Big)+\ff{c\e^{-n\aa
\tau}}{1-\kk}\\
&\le\kk\E\|X_{n\tau}(\xi)-X_{n\tau}(\eta)\|^2_\8+\kk\E\|X_{(n-1)\tau}(\xi)-X_{(n-1)\tau}(\eta)\|^2_\8+\ff{c\e^{-n\gg
\tau}}{1-\kk}.
\end{split}
\end{equation*}
That is, one has
\begin{equation*}
\E\|X_{n\tau}(\xi)-X_{n\tau}(\eta)\|_\8^2\le
\ff{\kk}{1-\kk}\E\|X_{(n-1)\tau}(\xi)-X_{(n-1)\tau}(\eta)\|^2_\8+\ff{c\e^{-n\gg
\tau}}{(1-\kk)^2}.
\end{equation*}
By an induction argument, we obtain that
\begin{equation}\label{q1}
\begin{split}
&\!\!\!\! \E\|X_{n\tau}(\xi)-X_{n\tau}(\eta)\|^2_\8\\
&\ \le
c\Big(\ff{\kk}{1-\kk}\Big)^n+\ff{c}{(1-\kk^2)}\Big\{\Big(\ff{\kk}{1-\kk}\Big)^{n-1}\e^{-\gg
\tau}
+\Big(\ff{\kk}{1-\kk}\Big)^{n-2}\e^{-2\gg
\tau}+\cdots+\e^{-n\gg \tau}\Big\}\\
&\ \le c\Big(\ff{\kk}{1-\kk}\Big)^n+\ff{\e^{-n\gg\tau}(1-q^n)}{1-q}\\
&\ \le c\e^{-pn\gg\tau}+\ff{\e^{-n\gg\tau}}{1-q}\\
&\ \le c\e^{-(p\wedge1)n\gg\tau},
\end{split}
\end{equation}
where
\begin{equation*}
p:=\ff{1}{\gg\tau}\log\Big(\ff{1-\kk}{\kk}\Big) \mbox{ and }
q:=\kk\e^{\aa\tau}/(1-\kk)<1
\end{equation*}
because $\kk<1/2$ and
$\aa$ can be taken sufficiently
small. Next,  for any $t>0$, note that there exists an $n\ge0$ such
that $t\in[n\tau,(n+1)\tau)$ and by \eqref{q1} that
\begin{equation}\label{q2}
\begin{split}
&\!\!\! \E\|X_t(\xi)-X_t(\eta)\|^2_\8 \\
&\ \le\E\|X_{n+1}(\xi)-X_{n+1}(\eta)\|^2_\8+\E\|X_n(\xi)-X_n(\eta)\|^2_\8\\
&\ \le
c\e^{-(p\wedge1)(n+1)\gg\tau}+c\e^{(p\wedge1)\gg\tau}\e^{-(p\wedge1)(n+1)\gg\tau}\\
&\ \le c\e^{-(p\wedge1)\gg t}.
\end{split}
\end{equation}
Recall that each bounded and continuous function on $\C$ may be
approximated pointwise by functions of $C_b^2(\C)$. By the
Burkhold-Davis-Gundy inequality and the Gronwall inequality, there
exists $\xi_n\in C_b^2(\C)$ such that
\begin{equation}\label{eq22}
\lim_{n\to\8}\E\|X_t(\xi)-X_t(\xi_n)\|^2_\8=0,\ \ \ t\ge0.
\end{equation}
For any $\xi,\eta\in U$, note that
\begin{equation}\label{eq23}
\begin{split}
\E\|X_t(\xi)-X_t(\eta)\|^2_\8&\le
2\E\|X_t(\xi_n)-X_t(\eta_n)\|^2_\8+4\E\|X_t(\xi)-X_t(\xi_n)\|^2_\8\\
&\quad+4\E\|X_t(\eta)-X_t(\eta_n)\|^2_\8,
\end{split}
\end{equation}
where $\eta_n\in C_b^2(\C)$ such that \eqref{eq22} with $\xi$ and $\xi_n$
replaced by $\eta$ and $\eta_n$, respectively. As a result, we conclude
 \eqref{N7} follows from \eqref{q2}, \eqref{eq22}, and \eqref{eq23}. Analogously, \eqref{N8} can be proved.
\end{proof}

\begin{rem}
{\rm Under   dissipative conditions, \cite[Theorem 4.2]{BY}
discusses existence of  stationary distributions for a class of
neutral SDEs. However, in this section, by the
variation-of-constants formula, we investigate existence and
uniqueness of stationary distributions for a range of semi-linear
SDEs of neutral type, which might not satisfy dissipative
conditions, see, e.g., Example \ref{nexa} below. }
\end{rem}

Finally,  we construct an example to demonstrate the theory
established in Theorem \ref{neutral}.

\begin{exa}\label{nexa}
{\rm Consider a linear neutral SDE
\begin{equation}\label{q3}
\d \Big(X(t)+\ff{1}{3}X(t-1)\Big)=-X(t-1)\d
t+a\int_{-1}^0X(t+\theta)\d\theta\d W(t),\ \ \ X_0=\xi,
\end{equation}
where $a\in\R$ and $\{W(t)\}_{t\ge0}$ is a real-valued Brownian
motion  defined on the probability space
$(\OO,\P,\F,\{\F_t\}_{t\ge0})$. The characteristic equation
associated with the deterministic counterpart of \eqref{q3} is
\begin{equation}\label{q4}
\ll+\Big(1+\ff{\ll}{3}\Big)\e^{-\ll}=0,\ \ \ \ll\in\mathbb{C}.
\end{equation}
A calculation by the MatLab shows that the unique root of \eqref{q4}
is $\ll=-2.313474269.$ Then, by Theorem \ref{neutral} we deduce that
\eqref{q3} has a unique stationary distribution if $a\in\R$ is
sufficiently small. }
\end{exa}

\section{Stationary Distributions for Retarded SDEs Driven by
Jump Processes}\label{Sec4} In the last two sections, we studied
existence and uniqueness of stationary distributions for retarded
SDEs with continuous sample paths. In this section, we turn to the
case of retarded SDEs driven by jump processes.
One distinct feature of the jump processes is
that
there may be no finite second moments.
As a result, there is no It\^o isometry that can be used.
The lack of the second moments and the jump discontinuity make
the problem more difficult to deal with.
Although the variation-of-constants approach can still be used,
the verification of the
tightness cannot be done as in the last two sections.
To overcome the difficulty, we use the Kurtz tightness criterion
to treat the underlying problem.

\smallskip
 Consider a retarded O-U process driven by a L\'{e}vy
process with the L\'{e}vy triple $(0,0,\nu)$  in the form
\begin{equation}\label{eq1}
\d X(t)=\Big(\int_{[-\tau,0]}X(t+\theta)\mu(\d\theta)\Big)\d t+\d
Z(t),\ \ \ X_0=\xi\in\D.
\end{equation}
The main result in this section is as follows.

\begin{thm}\label{jump}
{\rm Let $v_0<0$ and assume further that
\begin{equation}\label{c1}
\int_{|z|>1}|z|\nu(\d z)<\8.
\end{equation} Then  there is a unique stationary distribution $\pi\in\mathcal {P}(\D)$ for \eqref{eq1}. }
\end{thm}

\begin{proof}
 For each
integer $n\ge1$, set
\begin{equation*}
\mu_n(\cdot):=\ff{1}{n}\int_0^n\P(t,\xi,\cdot)\d t.
\end{equation*}
If $\{\mathcal {L}(X_t(\xi))\}_{t\ge\tau}$ is tight under the
Skorohod metric $\d_S$, for any $\vv>0$ there exists a compact
subset $U\in\B(\D)$ such that $$\P(X_t(\xi)\in U)\le1-\vv
\ \hbox{ and hence }\
 \mu_n(U)\le1-\vv.$$ Thus  $\{\mu_n(\cdot)\}_{n\ge1}$ is tight.
Recall  that $\{X_t(\xi)\}_{t\ge0}$ is Markovian by
\cite[Proposition 3.3]{RR} and eventually Feller, i.e., for all
$t\ge\tau$, $P_t$ maps $C_b(\D)$ into itself  due to
\cite[Proposition 3.5]{RR}. As a consequence, by the
Krylov-Bogoliubov theorem \cite[Theorem 3.1.1, p.21]{DZ}, we
conclude that \eqref{eq1} has a stationary distribution
$\pi\in\mathcal {P}(\D)$.

For $t>0$ and $\GG\in\mathscr{B}(\R\setminus\{0\})$, define the
Poisson random measure generated  by $Z(t)$ as
\begin{equation*}
N(t,\GG):=\sum_{s\in(0,t]}
{{\mathbf 1}}_{\GG}(\3 Z(s)),
\end{equation*}
where $$\bigtriangleup Z(t):=Z(t)-Z(t-) \ \hbox{ for } \ t\ge0
 \hbox{ with }\
Z(t-)=\lim_{s\uparrow t}Z(s),$$ and the compensated Poisson random
measure  by
\begin{equation*}
\tt N(t,\GG):=N(t,\GG)-t\nu(\GG).
\end{equation*}
By the L\'{e}vy-It\^o decomposition \cite[Theorem 2.4.16,
p.108]{A09}, one gets
\begin{equation}\label{eq11}
Z(t)=\int_{|z|\le1}z\tt N(t,\d z)+\int_{|z|>1}zN(t,\d z).
\end{equation}
 By
the variation-of-constants formula (see, e.g., Gushchin and
K\"{u}chler \cite{GK}), the solution of \eqref{eq1} can be written
explicitly as
\begin{equation}\label{eq6}
X(t;\xi)=\xi(0)r(t)+\int_{-\tau}^0\int_\theta^0
r(t+\theta-s)\xi(s)\d s\mu(\d \theta)+\int_0^tr(t-s)\d Z(s),
\end{equation}
where $\{r(t)\}_{t\ge-\tau}$ is the fundamental solution  of
\eqref{a1}. Substituting \eqref{eq11} into \eqref{eq6} leads to
\begin{equation*}
\begin{split}
X(t;\xi)=\xi(0)r(t)&+\int_{-\tau}^0\int_s^0r(t+s-u)\xi(u)\d u\mu(\d
s)\\
&
\quad+\int_0^t\int_{|z|\le1}r(t-s)z\tt N(\d s,\d
z)+\int_0^t\int_{|z|>1}r(t-s)zN(\d s,\d z)\\
&=:\sum_{j=1}^4I_j(t).
\end{split}
\end{equation*}
By \eqref{eq3}, it is easy to see that
\begin{equation*}
\sup_{t\ge0}\E(|I_1(t)|+|I_2(t)|)<\8.
\end{equation*}
Note from   the H\"older inequality, the It\^o isometry and
\eqref{eq3} that
\begin{equation*}
\begin{split}
\sup_{t\ge0}\E|I_3(t)|\le\bb^{1/2}\sup_{t\ge0}\Big(\int_0^t|r(t-s)|^2\d
s\Big)^{1/2}<\8,
\end{split}
\end{equation*}
where $$\bb:=\int_{|z|\le1}|z|^2\nu(\d z)<\8$$
since
$\nu(\cdot)$ is a L\'{e}vy measure.  Also, by \eqref{eq3} it follows
from \eqref{c1} that
\begin{equation*}
\sup_{t\ge0}\E|I_4(t)|\le\int_{|z|>1}|z|\nu(\d
z)\sup_{t\ge0}\Big(\int_0^t|r(t-s)|\d s\Big)<\8.
\end{equation*}
Hence we arrive at
\begin{equation}\label{eq7}
\dd:=\sup_{t\ge0}\E|X(t)|<\8.
\end{equation}
By \eqref{eq1} and \eqref{eq11}, for any $t\ge\tau$ we derive from
\eqref{eq7} that
\begin{equation}\label{w1}
\begin{split}
\E\|X_t(\xi)\|_\8 &\le\E|X(t-\tau;\xi)|+\E\Big(\sup_{t-\tau\le s\le
t}\Big|\int_{t-\tau}^s\int_{-\tau}^0X(u+\theta)\mu(\d\theta)\Big)\d
u\Big|\\
&\quad+\sup_{t-\tau\le s\le t}\Big|\int_{t-\tau}^s\int_{|z|\le1}z\tt
N(\d t,\d z)\Big|+\sup_{t-\tau\le s\le
t}\Big|\int_{t-\tau}^s\int_{|z|>1}zN(\d u,\d z)\Big|\Big)\\
&\le c+\E\Big(\sup_{t-\tau\le s\le
t}\Big|\int_{t-\tau}^s\int_{|z|\le1}z\tt N(\d u,\d
z)\Big|\Big)+\tau\int_{|z|>1}|z|\nu(\d z),
\end{split}
\end{equation}
where we have used that
\begin{equation*}
\E\Big(\sup_{t-\tau\le s\le t}\Big|\int_{t-\tau}^s\int_{|z|>1}zN(\d
u,\d z)\Big|\Big)\le\E\int_{t-\tau}^t\int_{|z|>1}|z|N(\d u,\d
z)=\tau\int_{|z|>1}|z|\nu(\d z).
\end{equation*}
Next, by the Burkhold-Davis-Gundy inequality (see, e.g.,
\cite[Theorem 48, p.193]{P04}) and the Jensen inequality, one finds
that
\begin{equation}\label{w2}
\begin{split}
&\!\!\! \E\Big(\sup_{t-\tau\le s\le t}\Big|\int_{t-\tau}^s\int_{|z|\le1}z\tt
N(\d u,\d z)\Big|\Big)\\
&\ \le c\E\ss{\int_{t-\tau}^t\int_{|z|\le1}
|z|^2N(\d u,\d z)}\\
&\ \le c\ss{\E\int_{t-\tau}^t\int_{|z|\le1} |z|^2N(\d u,\d
z)}\\
&\ =c\ss{\tau\int_{|z|\le1} |z|^2\nu(\d z)}.
\end{split}
\end{equation}
Consequently,  \eqref{w1} and  \eqref{w2} yield that
\begin{equation}\label{w3}
\sup_{t\ge0}\E\|X_t(\xi)\|_\8<\8.
\end{equation}
 Set
$$\E_s\cdot:=\E(\cdot|\F_s),\ \ s\ge0.$$
For $\theta\in[-\tau,0]$ and $\wdt \theta\in[0,\3]$, where $\3>0$ is
an arbitrary constant such that $\theta+\3\in[-\tau,0]$,
 \eqref{eq1} and
\eqref{eq11} lead to
\begin{equation*}
\begin{split} &\!\!\!
\E_{t+\theta}|X_t(\theta+\wdt \theta)-X_t(\theta)|\\ &
\ =\E_{t+\theta}|X(t+\theta+\wdt \theta)-X(t+\theta)|\\
& \ \le
\int_{t+\theta}^{t+\theta+\3}\E_{t+\theta}\Big\{\Big|\int_{[-\tau,0]}X(s+\theta)\mu(\d\theta)\Big|
+\int_{|z|\le1}z\tt N(\d s,\d z)+\int_{|z|>1}zN(\d s,\d
z)\Big\}\d s.
\end{split}
\end{equation*}
It follows that there is a $\gamma_0 (t,\3)$ satisfying
$$\E_{t+\theta}|X(t+\theta+\wdt \theta)-X(t+\theta)| \le \E_{t+\theta}\gamma_0(t,\3).$$
Taking expectation and $\limsup_{t\to \infty}$ followed by
$\lim_{\3\to 0}$, we obtain from the H\"older inequality and
\eqref{eq7} that
\begin{equation}\label{w4}
\lim_{\3\to0}\limsup_{t\to \infty}\E\gamma_0(t,\3)=0.
\end{equation}
In view of   Lemma \ref{tightness} and the time shift $t\mapsto
t-\tau$, we conclude from \eqref{w3} and \eqref{w4} that
$\{X_t(\xi)\}_{t\ge0}$ is tight under the Skorohod metric $\d_S$.
Finally, the desired assertion follows by combining the arguments of
Theorem \ref{diffusion} and that of Theorem \ref{jump}. Next, from
\eqref{eq3} and \eqref{eq6}
 we get that
\begin{equation}\label{eq13}
\begin{split}
\E\|X_t(\xi)-X_t(\eta)\|_\8
&=\E\Big(\sup_{-\tau\le\theta\le0}\Big|(\xi(0)-\eta(0))r(t+\theta)\\
&\qquad\qquad+\int_{[-\tau,0]}\int_u^0 r(t+\theta+u-s)(\xi(s)-\eta(s))\d
s\mu(\d u)\Big|\Big)\\
&\le c\e^{-\gg t}\|\xi-\eta\|_\8.
\end{split}
\end{equation}
This yields that
\begin{equation}\label{c2}
\lim_{t\to\8}\|\P(t,\xi,\cdot)-\P(t,\eta,\cdot)\|_{\mbox{var}}=\lim_{t\to\8}\sup_{\mbox{Lip}(\varphi)=1}|\E\varphi(X_t(\xi))-\E\varphi(X_t(\eta))|=0,
\end{equation}
where $\|\cdot\|_{\mbox{var}}$ denotes the total variation of a
signed measure and $\mbox{Lip}(\varphi)$ is the Lipschitz constant
of $\varphi$ with respect to the Skorohod metric $\d_S.$   If
$\pi^\prime(\cdot)\in\mathcal {P}(\D)$ is also a stationary
distribution, then, by the invariance, one has
\begin{equation}\label{c3}
\|\pi-\pi^\prime\|_{\mbox{var}}\le\int_{\D\times\D}\|\P(t,\xi,\cdot)-\P(t,\eta,\cdot)\|_{\mbox{var}}\pi(\d\xi)\pi(\d\eta)
\end{equation}
Thus, the uniqueness of stationary distribution follows from
\eqref{c2} and by taking $t\to\8$ in  \eqref{c3}.
\end{proof}

\begin{rem}
{\rm  By \eqref{c1},   $\E|Z(t)|<\8$ for all $t>0$
because of \cite[Theorem 2.5.2, p.132]{A09}.  \eqref{eq1}
incorporates retarded O-U processes driven by symmetric $\aa$-stable
processes, which have finite $p$th moment with $p\in(0,\aa)$, and
subordinate Brownian motions $W_{S(t)}$, i.e.,  $W(t)$ is a standard
Brownian motion and $S(t)$ is an $\aa/2$-stable subordinator (i.e.,
a real-valued L\'{e}vy process with nondecreasing sample paths). For
more details on stable distributions and subordinator, we refer to
Applebaum \cite[p.33-62]{A09}.  }
\end{rem}

\begin{rem}{\rm
In many applications, one often encounters the so-called jump diffusion models,
in which both Brownian type of noise and L\'evy process appear. In view of our
results in Section \ref{Sec2} and the current section, in lieu of \eqref{a5} or \eqref{eq1},
we can consider a process of the form
\begin{equation}\label{a5a}
\d X(t)=\Big(\int_{-\tau}^0X(t+\theta)\mu(\d\theta)\Big)\d
t+\si(X_t)\d W(t) +dZ(t),\ \ \ X_0=\xi\in\D.
\end{equation}
Continue to use the variation-of-constants formula.
Comparing to the development in Theorem \ref{jump}, we need to deal with
an additional term involving conditional expectation of an  integral in the verification of tightness.
This term can be easily handled by use of the Cauchy-Schwarz inequality and properties of Brownian motion.
The results of Theorem \ref{jump} continue to hold.
}\end{rem}

 \begin{rem}
 {\rm
 Examining the proof
 of Theorem \ref{jump}, the technique
employed therein applies to \eqref{eq1} with the L\'{e}vy triple
$(0,a,\nu)$ and a retarded SDE with jumps
\begin{equation}\label{c4}
\d X(t)=\Big(\int_{-\tau}^0X(t+\theta)\mu(\d\theta)\Big)\d
t+\si(X_{t-})\d Z(t),\ \ \ X_0=\xi\in\D,
\end{equation}
where $\si:\D\to\R$ is uniformly bounded  and
$$X_{t-}(\theta):=\lim_{s\uparrow t+\theta}X(s), \ \ \theta\in[-\tau,0].$$
}
\end{rem}

If the L\'{e}vy process $Z(t)$ has a finite second moment, the
uniform boundedness of $\si$ can indeed be removed as the following
theorem shows.

\begin{thm}
{\rm Let $v_0<0$ and \eqref{a6} hold  for a sufficiently small $L>0$
and arbitrary $\xi,\eta\in\D$,  and suppose further that
\begin{equation}\label{c5}
\int_{|z|\ge1}|z|^2\nu(\d z)<\8.
\end{equation}
Then \eqref{c4} has a unique stationary distribution $\pi\in\mathcal
{P}(\D)$. }
\end{thm}

\begin{proof}
By the variation-of-constants formula (see, e.g., \cite[Theorem
3.1]{RR}), one has
\begin{equation*}
X(t;\xi)=\xi(0)r(t)+\int_{-\tau}^0\int_\theta r(t+\theta-s)\xi(s)\d
s\mu(\d \theta)+\int_0^tr(t-s)\si(X_{s-})\d Z(s),
\end{equation*}
where $\{r(t)\}_{t\ge-\tau}$ is the fundamental solution  to
\eqref{a1}. By \eqref{c5},   the L\'{e}vy-It\^o decomposition (see,
e.g., \cite[Theorem 2.4.16, p.126]{A09}) gives that
\begin{equation}\label{a8}
Z(t)=at+W(t)+\int_{z\neq0}z\tt N(t,\d z),\ \ \ a\in\R.
\end{equation}
Thus, we have
\begin{equation*}
\begin{split}
X(t;\xi)&=\xi(0)r(t)+\int_{-\tau}^0\int_\theta r(t+\theta-s)\xi(s)\d
s\mu(\d \theta)+a\int_0^tr(t-s)\si(X_{s-})\d
s\\
&\quad+\int_0^tr(t-s)\si(X_{s-})\d
W(s)+\int_0^t\int_{z\neq0}r(t-s)\si(X_{s-})z\tt N(\d s,\d z).
\end{split}
\end{equation*}
Carrying out a similar argument to that of  \eqref{a3} and taking
\eqref{c5} into account, we can also deduce that
\begin{equation}\label{a4}
\sup_{t\ge0}\E|X(t,\xi)|^2<\8.
\end{equation}
From \eqref{c4} and \eqref{a8},  ones derive from \eqref{a4} that
\begin{equation*}
\begin{split}
\E\|X_t(\xi)\|_\8^2&\le c\Big\{1+\E\Big(\sup_{t-\tau\le s\le
t}\Big|\int_{t-\tau}^s\si(X_{s})\d
W(s)\Big|^2\Big)\\
&\qquad\quad+\E\Big(\sup_{t-\tau\le s\le
t}\Big|\int_{t-\tau}^s\int_{z\neq0}\si(X_{s-})z\tt N(\d s,\d
z)\Big|^2\Big)\Big\}\\
&=:c\{1+\Gamma_1(t)+\GG_2(t)\}.
\end{split}
\end{equation*}
Also, by the Burkhold-Davis-Gundy inequality (see, e.g.,
\cite[Theorem 48, p.193]{P04}), from \eqref{a6} and \eqref{a4} we
obtain that
\begin{equation}\label{w5}
\begin{split}
\GG_2(t) &\le c\E\int_{t-\tau}^t\int_{z\neq0} |
\si(X_{s-})|^2|z|^2N(\d s,\d z)\le c.
\end{split}
\end{equation}
Then, \eqref{a6}, \eqref{w5} and an application of the
Burkhold-Davis-Gundy inequality (see, e.g., \cite[Theorem 7.3,
p.40]{M08}) applies to $\GG_1(t)$ imply that
\begin{equation*}
\sup_{t\ge0}\E\|X_t(\xi)\|_\8^2<\8.
\end{equation*}
Finally, the desired assertion follows by  imitating the argument of
Theorem \ref{jump}.
\end{proof}

\begin{rem}
{\rm For the case that $\si:\D\to\R$ is uniformly bounded, Rei$\bb$
et al. \cite{RR} explored existence of a stationary distribution of
\eqref{c4} by considering the semi-martingale characteristics to
show the tightness of the segment processes. In their paper,
it was mentioned that ``For the latter the imposed
boundedness of $F$ can certainly be relaxed considered, but will
then depend on the large jumps of $L$, that is, on fine properties
of $\nu.$'' In this section, we give a positive answer to this
problem by
virtue of Kurtz's tightness criterion.
}
\end{rem}

\begin{rem}
{\rm  In this paper, for
notational simplicity, we only treated the
existence and uniqueness of stationary distributions for several
classes of real-valued retarded SDEs without dissipativity.
Our results can  be readily generalized to the multidimensional
cases.  The key is the use of a multidimensional variation-of-constants formula.
For the corresponding finite dimensional and infinite dimensional variation-of-constants
formulas,
 we refer the reader to
\cite[Chapter 6 and Chapter 9]{HL} and \cite{L,L10}, respectively.
}
\end{rem}

\end{document}